\documentclass[11pt, leqno]{amsart}

\usepackage[utf8]{inputenc}
\usepackage[T1]{fontenc}
\usepackage[all]{xy}
\usepackage{t1enc}
\usepackage{amsfonts,amsmath,amssymb,amsthm,color,amsbsy}
\usepackage{color}

\usepackage{bbm}
\usepackage{eucal}
\usepackage{enumerate}
\usepackage{graphicx}
\usepackage{pstricks-add}
\usepackage{pst-plot}
\usepackage{float}
\usepackage{pst-node}
\usepackage{url}
\usepackage{paralist}
\usepackage{mathrsfs}
\usepackage{color}
\usepackage{colortbl}
\usepackage{bm}
\usepackage{mathtools}

\newtheorem{thm}{Theorem}[section]

\theoremstyle{definition}
\newtheorem{df}[thm]{Definition}
\newtheorem{rem}[thm]{Remark}

\def\d{\mathrm{d}}

\def\f{\varphi}

\def\e{\varepsilon}
\def\pr{\right )}
\def\le{\left (}

\def\R{\mathbb{R}}

\def\Z{\mathbb{Z}}
\def\mm{\kern +2pt \raisebox{+0.5 pt}{$\shortmid$}\kern -2pt\hbox{$\multimap$}\kern +2pt}
\def\H1{H^1\le\Omega,\mathbb{R}^N\pr}
\def\h2{H^2\le\Omega,\mathbb{R}^N\pr}
\def\h10{H^1_0\le\Omega,\mathbb{R}^N\pr}
\def\Pr{\mathrm{Pr}}
\def\o{\omega}
\def\t{\theta}
\def\Ra{\mathrm{Ra}}
\def\rot{\mathrm{rot}\,}
\def\div{\mathrm{div}\,}

\def\ov{\overline}
\def\wh{\widehat}

\author{Grzegorz \L ukaszewicz}
\address{University of Warsaw, Mathematics Department, ul.Banacha 2, 02-957 Warsaw, Poland}
\email{glukasz@mimuw.edu.pl}

\author{Jakub Siemianowski}
\address{Faculty of Mathematics and Computer Sciences, Nicolaus Copernicus University, Chopina 12/18, 87-100 Toru\'n, Poland}
\email{jsiem@mat.umk.pl}

\title{Existence of Solutions for Thermomicropolar Fluid}

\date{}

\begin{document}
\maketitle

\section{Introduction}
The theory of micropolar fluids is a generalization of the Navier-Stokes model in the sense that it takes into account the microstructure of the fluid.
The theory is expected to provide a~mathematical model for the non-Newtonian fluid behaviour observed in certain fluids such as~polymers, colloidal fluids, liquid crystals, blood, ferro liquid, real fluids with suspensions, which is~more realistic.
It~was introduced by Eringen in \cite{eringen}.
The~theory of thermomicropolar fluids, proposed by Eringen in \cite{eringen2}, extends the theory of micropolar fluid by including the heat conduction and heat dissipation effects.

\section{Formulation of the problem}
We consider a certain simplification of the two-dimensional  thermomicropolar fluid equations introduced in \cite{eringen2}.
The governing system of equations describing thermomicropolar fluids, after nondimentionalization, is 
\begin{equation}\label{row.thermomicropolar}
\begin{aligned}
\frac{1}{\Pr} \le u_t + \le u \cdot \nabla \pr u \pr + \nabla p &= \Delta u + 2N^2 \rot\o + e_2\Ra T,\\
\div u&=0,\\
\frac{1}{\Pr} \le \o_t + \le u\cdot \nabla \pr \o \pr + 4N^2\o &= \frac{1}{L^2}\Delta \o + 2N^2 \rot u,\\
T_t + u \cdot \nabla T &= \Delta T  + D \rot \o \cdot \nabla T,
\end{aligned}
\end{equation}
where $u=(u_1,u_2)$ is the velocity field, $p$ is the pressure, $\o$ is the microrotation and $e_2$ is the unit upward vector $(0,1)\in\R^2$.
The positive constants $N^2<1$, $L^2$, $D$ are related with viscosity coefficients, the Prandtl number $\Pr$ describes the relative importance of kinematic viscosity over thermal diffusivity and the Rayleigh number $\Ra$ measuring measuring the ratio of overall buoyancy force to the damping coefficients.
To bring no confusion we explain
\[
\div u = \dfrac{\partial u_1}{\partial x_1}+\dfrac{\partial u_2}{\partial x_2},\quad \rot u := \dfrac{\partial u_2}{\partial x_1}-\dfrac{\partial u_1}{\partial x_2},\quad \rot \o := \le \dfrac{\partial \o}{\partial x_2}, - \dfrac{\partial \o}{\partial x_1}\pr.
\]
We assume that the fluids occupy the (nondimensionalized) region 
\[
\Omega_\infty:= (-\infty,\infty)\times (0,1).
\]
The system \eqref{row.thermomicropolar} is equipped with the following boundary conditions
\[
u=0\restriction_{x_2 = 0,1},\quad\o\restriction_{x_2=0,1}=0,\quad T\restriction_{x_2=0}=1\quad\text{and}\quad T\restriction_{x_2=1} = 0
\]
with $l$-periodicity in the $x_1$-direction assumed.
The initial conditions are 
\[
u\restriction_{t=0} = u_0,\quad\o\restriction_{t=0} = \o_0,\quad T\restriction_{t=0} = T_0\quad\text{for }x=(x_1,x_2)\in\Omega_\infty.
\]

The system is a model for convection of micropolar fluids of a layer of fluids bounded by two horizontal one-dimensional parallel plates a distance $h$ apart with the bottom heated at temeprature $T_0$ and the top cooled at temperature $T_1<T_0$.
Therefore, the fluid motion is induced by differential heating.

\section{Preliminaries}
We define $\Omega\subset \Omega_\infty$ to be a rectangle box of length equal to the period
\[
\Omega:= (0,l)\times (0,1)
\]
and divide the boundary $\partial \Omega$ into two parts
\begin{gather*}
\Gamma_h:= \left \{ (x_1,x_2)\in \partial\Omega\mid x_2=0,1\right \},\\
\Gamma_l:= \left \{ (x_1,x_2)\in \partial\Omega\mid x_1=0\text{ or }x_1=l\right \}
\end{gather*}
We introduce
\begin{gather*}
\widetilde{V}_S:= \left \{ u \in C^\infty \le \Omega_\infty\pr^2 \mid \div u = 0, \, u\text{ is }l\text{-periodic in }x_1\text{-direction},\;u\restriction_{\Gamma_h}=0\text{ (as a limit)}\right\},
\\
\widetilde{V}:= \left \{\o \in C^\infty (\Omega) \mid   \o\text{ is }l\text{-periodic in }x_1\text{-direction},\;\o\restriction_{\Gamma_h}=0 \right\},\\
H_S:=\text{ closure of }\widetilde{V}_S\text{ in }L^2(\Omega)^2,\\
H:=L^2(\Omega),\text{ because }C^\infty _0(\Omega)\subset \widetilde{V},\\
V_S:=\text{ closure of }\widetilde{V}_S\text{ in }H^1(\Omega)^2\text{ and }\\
V:=\text{ closure of }\widetilde{V}\text{ in }H^1(\Omega).
\end{gather*}
Spaces $H_i$ are Hilbert spaces with the inner products
\[
(u,v):= (u,v)_{L^2}=\int_\Omega u(x)\cdot v(x)\d x, \quad u,v\in H_i, \; i=L,S.
\]
We have the Poincar\'{e} inequality (see \cite[pp. 51--52]{temam})
\[
|v|\leq k_1\|v\|,\quad\text{for }v\in V, V_S.
\]
Thus, $V$ and $V_S$ are Hilbert spaces when equipped with the inner product
\[
((u,v)) = \int_\Omega\nabla u\cdot \nabla v \quad u,v\in V_i,\; i=L,S.
\]
We denote the corresponding norms by
\[
|v|:= (v,v) ^{1/2}\quad \text{for }v\in H_i,\; i=L,S
\]
and
\[
\|v\| = ((v,v))^{1/2},\quad\text{for }v\in V_i,\; i =L,S.
\]

We define the standard trilinear forms (see \cite{robinson} or \cite{temam})
\begin{gather*}
b_S(u,v,w):=\sum_{i,j=1}^2\int_\Omega u_i\frac{\partial v_j}{\partial x_i}w_j\qquad\text{and}\qquad
b(u,v,w):=\sum_{i=1}^2\int_\Omega u_i\frac{\partial v}{\partial x_i}w.
\end{gather*}
We introduce the Laplace operator associated with our boundary conditions.
Let
\begin{equation}\label{wzor.dziedzina.laplasjanu}
D(A):= \left \{ v\in V \mid -\Delta v \in H \right \}
\end{equation}
and define $A:D(A)\to H$ by
\[
Au := -\Delta u,\quad u\in D(A). 
\]
Clearly, the eigenvectors $\{v_k\}_{k\geq 1}\subset D(A)$ of $A$ forms the orthonomal basis of $H$ and we have
\begin{equation}\label{laplasjan.wart.wekt.wl}
A v_k = \beta_k v_k \quad\text{and}\quad 0<\beta_1\leq \beta_2 \leq \ldots\leq \beta_k \to \infty.
\end{equation}

Since the domain $\Omega$ is a rectangle, these eigenvectors and eigenvalues can be explicitly determined: for $n\in\Z$, $m\geq 1$ we have
\begin{equation}\label{wart.wekt.wl}
\begin{aligned}
\beta_{nm}&=\le \frac{2n\pi}{l}\pr^2 + (2m\pi)^2,\\
v_{nm}&=\sqrt{\frac{2}{l}}\left [\sin \le \frac{2n\pi}{l}x \pr+\cos \le \frac{2n\pi}{l}x\pr\right ] \sin(m\pi y).
\end{aligned}
\end{equation}
It should be stressed that each $v_{nm}$ belongs to $C^\infty (\overline{\Omega})$.\footnote{This means $v_{nm}\in C^\infty(\Omega)$ and each derivative $D^\alpha v_{nm}$ admits continuous extnsion to $\overline{\Omega}$.}
We can renumber them so that $(v_k)_{k\geq 1}$ satisfy \eqref{laplasjan.wart.wekt.wl}.

Since we know the exact formulas for eigenfunctions we can prove the following results.
\begin{thm}[Regularity theorem]\label{tw.o.reg}
Let $ f\in H$ and let $u\in D(A)$ satisfy
\begin{equation}\label{row.reg}
A u = f.
\end{equation}
Then $u\in H^2(\Omega)$ and $\|u\|_{H^2}\leq C|f|$ for some constant $C>0$.
As a result, norms $\|u\|_{H^2}$ and $|Au|$ are equivalent on $D(A)$.
\end{thm}

For our purposes, we introduce the fractional power of the Laplacian.
Let 
\[
D(A^{3/2}) = \left\{u\in H\mid u = \sum_{k\geq 1} (u,v_k) v_k\text{ (in }H), \sum_{k\geq 1} (u,v_k)^2 \beta_k^3 <\infty\right\},
\]
and define
\[
A^{3/2} u = \sum_{k\geq 1} \beta_k^{3/2}(u,v_k) v_k.
\]
To make $D(A^{3/2})$ a Hilbert space we equip it with the inner product
\[
(u,v)_{D(A^{3/2})} = (A^{3/2}u,A^{3/2}v)
\]
which gives the corresponding norm 
\[
\|u\|_{D(A^{3/2})} = |A^{3/2} u|.
\]

\begin{thm}\label{tw.ulamkowy.laplasjan}
We have
\[
D(A^{3/2})\subset H^3(\Omega)
\]
and norms $\|\cdot\|_{H^3}$, $\|\cdot\|_{D(A^{3/2})}$ are equivalent on $D(A^{3/2})$.
\end{thm}

Now, we introduce the Stokes operator $A_S$, see \cite{robinson} or \cite{temam} for details.
Let 
\[
D(A_S):= \left\{u\in V_S \mid \exists w\in H_S,\; \forall \f\in V_S\; (w,\f) = ((u,\f)) \right\}
\]
and define
\[
A_S(u) = w.
\]
It is known that $A_S=-P\Delta$, where $P$ stands for the Helmholtz--Leray projector from $L^2(\Omega)^2$ onto $H_S$.

We present two versions of the Gagliardo-Nirenberg inequality --- see \cite[Thm 5.8]{adams_fournier}:
\begin{gather}
\|v\|_{L^4}\leq k_2 |v|^{1/2}\|v\|_{H^1}^{1/2},\qquad\text{for }v\in H^1(\Omega),\label{G.N.4}\\
\|v\|_{L^\infty}\leq k_3|v|^{1/2}\|v\|^{1/2}_{H^2},\qquad\text{for }v\in H^2(\Omega).\label{G.N.infty}
\end{gather}
Combinig \eqref{G.N.4} with the Poincar\'e inequlity yields the so-called Ladyzhenskaya's inequality
\begin{equation}\label{ladyzhenskaya}
\|v\|_{L^4}\leq k_4 |v|^{1/2}\|v\|^{1/2},\qquad \text{for }v\in V,\,V_S,
\end{equation}
and combining \eqref{G.N.4} with Theorem \ref{tw.o.reg} gives
\begin{equation}\label{nier.grad}
\|\nabla v\|_{L^4}\leq k_5\|v\|^{1/2}|Av|^{1/2},\qquad v\in D(A).
\end{equation}
Similarly, combining \eqref{G.N.infty} with the regularity theorems for the Laplace (Theorem \ref{tw.o.reg}) operator and for the Stokes operator (see \cite[pp. 835--836]{bouk_luk_real}) we get
\begin{gather}
\|v\|_{L^\infty}\leq k_6|v|^{1/2}|Av|^{1/2},\qquad\text{for }v\in D(A),\label{Agmon}\\
\|u\|_{L^\infty}\leq k_7|u|^{1/2}|A_Su|^{1/2},\qquad\text{for }u\in D(A_S).\label{Agmon.S}
\end{gather}

\section{Initial data in $L^2$}
Following a traditional approach, we change the temperature equation from \eqref{row.thermomicropolar} so that the perturbative variable will satisfy the homogeneous boundary conditions:
\[
\t:= T - (1-x_2).
\]
We also change the pressure $p$ to $p-\Ra (x_2-x_2^2)$ in the velocity equation in \eqref{row.thermomicropolar}.
These transforms \eqref{row.thermomicropolar} into
\begin{gather}
\frac{1}{\Pr} \le u_t + \le u \cdot \nabla \pr u \pr + \nabla p = \Delta u + 2N^2 \rot\o + e_2\Ra \t,\label{row.u.2}\\
\div u=0,\label{row.incompr.2}\\
\frac{1}{\Pr} \le \o_t + \le u\cdot \nabla \pr \o \pr + 4N^2\o = \frac{1}{L^2}\Delta \o + 2N^2 \rot u,\label{row.omega.2}\\
\t_t + u \cdot \nabla \t = \Delta \t  + D \rot \o \cdot \nabla \t + D\frac{\partial \o}{\partial x_1}+u_2\label{row.teta}
\end{gather}
on $\Omega\times(0,\infty)$, equipped with the boundary conditions
\begin{equation}\label{war.brzeg}
u\restriction_{\Gamma_h}=0,\quad
\o\restriction_{\Gamma_h}=0,\quad
\t\restriction_{\Gamma_h}=0
\end{equation}
and periodicity in the horizontal direction.
The initial conditions now read
\begin{equation}\label{war.pocz}
u(0) = u_0,\quad\o (0) = \o_0,\quad \t(0) = \t_0 = T_0-(1-x_2).
\end{equation}

\begin{df}
Let $T>0$,  $u_0\in H_S$, $\o_0\in H$ and $\t_0\in H$.
By a weak solution of problems \eqref{row.u.2}--\eqref{war.pocz} we mean a triple of functions $(u,\o,\t)$,
\begin{equation}\label{warunk.ist.L2}
\begin{aligned}
u&\in L^2(0,T;V_S)\cap C([0,T],H_S)\cap W^{1,2}(0,T;V_S^\ast),\\
\o&\in L^2(0,T;V)\cap C([0,T],H)\cap W^{1,2}(0,T;V^\ast),\\
\t&\in L^2(0,T;V)\cap L^\infty(0,T;H)\cap W^{1,2}(0,T;D(A^{3/2})^\ast)
\end{aligned}
\end{equation}
such that $u(0)=u_0$, $\o(0)=\o_0$, $\t(0)=\t_0$ and satisfying the following identities
\[
\frac{1}{\Pr}\le \frac{d}{dt}(u(t),\f) + b_S(u(t),u(t),\f)\pr + (\nabla u(t),\nabla \f) = 2N^2 (\rot \o(t),\f) + \Ra (\t(t) e_2,\f)
\]
for every $\f\in V_S$,
\[
\frac{1}{\Pr}\le \frac{d}{dt}(\o(t),\psi) + b(u(t),\o(t),\psi)\pr + 4N^2 (\o(t),\psi) + \frac{1}{L^2}(\nabla \o(t),\nabla \psi) = 2N^2 (\rot u(t),\psi)
\]
for every $\psi\in V$,
\[
\frac{d}{dt}(\t(t),\eta) + b(u(t),\t(t),\eta) + (\nabla\t(t),\nabla\eta)  = -D (\t(t),\rot \o(t) \cdot \nabla \eta)+ D \le\frac{\partial\o}{\partial x_1}(t),\eta\pr + (u_2,\eta)
\]
for every $\eta\in D(A^{3/2})$, in the sense of scalar distributions on $(0,T)$.
\end{df}

\begin{rem}
Identifying $H$ with its dual $H^\ast$ via the Riesz isomorphism leads to the following embeddings
\[
D(A^{3/2})\subset D(A) \subset V\subset H\simeq H^\ast \subset V^\ast \subset D(A)^\ast \subset D(A^{3/2})^\ast.
\]
Because $\t\in W^{1,2}(0,T;D(A^{3/2})^\ast)$ implies $\t\in C([0,T],D(A^{3/2})^\ast)$, the condition $\t(0)=\t_0$ is meaningful.
\end{rem}

\begin{thm}\label{tw.istnienie.L2}
Let $u_0\in H_S$, $\o_0$, $\t_0\in H$, and $T>0$.
There is a weak solution $(u,\o,\t)$ to \eqref{row.u.2}--\eqref{war.pocz}.
\end{thm}

\begin{proof}
In what follows we use the so-called Galerkin approximations of \eqref{row.u.2}---\eqref{row.teta}.
As it is standard, we will show only the a priori estimates.
The reader unfamiliar with this technique is encouraged to see \cite{robinson} or \cite{temam}.

Take the inner product in $\R^2$  of \eqref{row.u.2} and $u$, and integrate over $\Omega$
\begin{equation}\label{oszacowanie.energ.u}
\frac{1}{\Pr}\le \frac{\partial}{\partial t}u ,u \pr   + \frac{1}{\Pr} b_S(u, u, u) + \le A_S u,u\pr = 2N^2 \le \rot \o, u \pr + \Ra \le\t e_2, u \pr 
\end{equation}
We have 
\[
\le \frac{\partial}{\partial t}u ,u \pr = \frac{1}{2}\frac{d}{dt} |u |^2,\qquad
\le A_S u, u \pr = \|u\|^2.
\]
Since 
\[
b_S(u,v,w)=-b_S(u,w,v),\quad u,v,w\in V_S,
\] 
we obtain 
\begin{equation}
b_S(u,u,u)=0.
\end{equation}
We integrate by parts, use the Cauchy-Schwarz inequality and Young's inequality
\begin{equation}
2N^2\le \rot \o, u \pr = 2N^2\le \o,\rot u \pr \leq 2N^2|\o||\rot u| \leq 2N^2|\o|^2 + \frac{N^2}{2}|\rot u|^2  
\end{equation}
and, because of $\div u =0$, we get 
\[
\rot\rot u =-\Delta u.
\]
Hence, we obtain
\[
2N^2(\rot \o,u)\leq 2N^2|\o|^2 + \frac{N^2}{2}\|u\|^2,
\]
and similarly
\begin{equation}
\Ra\le \t e_2 , u \pr \leq \Ra|\t||u| \leq \frac{k_1^2\Ra }{2(1-N^2)} |\t|^2 +\frac{1-N^2}{2k_1^2\Ra}|u|^2\leq \frac{k_1^2\Ra}{2(1-N^2)} |\t|^2 +\frac{1-N^2}{2\Ra}\|u\|^2,
\end{equation}
where we used the Poincar\'e inequality.
As a result \eqref{oszacowanie.energ.u} turns into
\begin{equation}\label{oszacowanie.energ.u2}
\frac{1}{\Pr}\frac{d}{dt} |u |^2 + \|u\|^2 \leq 4N^2 |\o|^2 + \frac{k_1^2\Ra^2}{1-N^2} |\t|^2.
\end{equation}

Multiply \eqref{row.omega.2} by $\o$ and integrate over $\Omega$
\begin{equation}\label{oszacowanie.energ.omega}
\frac{1}{\Pr}\le \frac{\partial}{\partial t}\o , \o \pr + \frac{1}{\Pr} b(u, \o, \o) + 4N^2 |\o|^2 + \frac{1}{L^2}\le A \o ,\o \pr = 2N^2 \le \rot u, \o \pr.
\end{equation}
As above, we have
\[
\le \frac{\partial}{\partial t}\o , \o \pr =\frac{1}{2}\frac{d}{dt}|\o|^2,\qquad
 b(u, \o, \o ) =0,\qquad
\le A \o ,\o \pr = \|\o\|^2 .
\]
and
\[
2N^2\le  \rot u, \o \pr \leq 2N^2|\rot u| |\o| \leq 2N^2|\o|^2 + \frac{N^2}{2}\|u\|^2.
\]
From the above and \eqref{oszacowanie.energ.omega}, we obtain
\begin{equation}\label{oszacowanie.energ.omega2}
\frac{1}{Pr}\frac{d}{dt}|\o|^2 +\frac{2}{L^2}\|\o\|^2 +4N^2|\o|^2 \leq N^2 \|u\|^2.
\end{equation}

Multiply \eqref{row.teta} by $\t$ and integrate over $\Omega$
\begin{equation}\label{row.teta.2}
\le\frac{\partial }{\partial t}\t, \t\pr   + b(u,\t,\t)  + \le A \t, \t\pr = D\int_\Omega (\rot\o\cdot \nabla \t ) \t + D\le \frac{\partial \o}{\partial x_1},\t\pr + \le u_2, \t\pr.
\end{equation}
We have
\[
\le\frac{\partial }{\partial t}\t, \t\pr = \frac{1}{2} \frac{d}{dt}\left | \t \right | ^2,\qquad
b(u,\t,\t) = 0,\qquad
\le A \t, \t \pr = \| \t \|^2,
\]
and we integrate by parts to get
\[
 \le \frac{\partial \o}{\partial x_1},\t\pr  = \le \o, \frac{\partial \t}{\partial x_1}\pr \leq |\o |\left |\frac{\partial \t}{\partial x_1} \right | \leq |\o | \| \t \| \leq D |\o|^2 + \frac{1}{4D}\|\t\|^2.
\]
We estimate
\[
( u_2, \t ) \leq |(u)_2| |\t| \leq |u| |\t| \leq |u|^2 + \frac{1}{4}|\t|^2.
\]
We assume for the moment that $\t$ is smooth enough so that the integral
\[
\int_\Omega \le\rot \o\cdot \nabla  \t \pr \t
\]
is well-defined.
In fact, we estimate the Galerkin-Faedo approximations $\t_m(x,t)= \sum_{j=1}^m \t_{mj}(t) v_j(x)$ of $\t$ which are of class $C^\infty(\overline{\Omega})$ with respect to the varibale $x$, see \eqref{wart.wekt.wl}.
Integrating by parts leads to
\begin{multline*}
\int_\Omega \le\rot \o\cdot \nabla  \t \pr \t = \frac{1}{2} \int_\Omega \le \frac{\partial}{\partial x_2}\o,-\frac{\partial}{\partial x_1}\o\pr \cdot \nabla \le \t ^2\pr  \\
=-\frac{1}{2}\int_\Omega \o\rot\le \nabla \le \t ^2 \pr\pr + \frac{1}{2}\int_{\partial\Omega}\o \le\frac{\partial \t}{\partial x_1} \t n_2 - \frac{\partial \t}{\partial x_2} \t n_1 \pr\d S =: I_1 + I_2
\end{multline*}
The symmetry of second derivatives implies $I_1 = 0$.
Since each $\dfrac{\partial v_k}{\partial x_2}$ is $l$-periodic in $x_1$ (see \eqref{wart.wekt.wl}), the boundary conditions yields $I_2=0$.
Therefore, \eqref{row.teta.2} leads to
\begin{equation}\label{oszacowanie.energ.teta}
 \frac{d}{dt}\left | \t \right | ^2 +  \| \t \|^2 \leq  2D^2|\o|^2 + 2|u|^2 .
\end{equation}

Add \eqref{oszacowanie.energ.u2}, \eqref{oszacowanie.energ.omega2} and \eqref{oszacowanie.energ.teta}
\begin{equation}\label{oszacowanie.energ.uomegateta}
\frac{d}{dt}\le |u|^2 + |\o|^2 + |\t|^2\pr + c_2 \le \|u\|^2 + \|\o\|^2 +\|\t\|^2\pr \leq c_3 \le |u|^2 + |\o |^2 + |\t|^2 \pr,
\end{equation}
where 
\[
c_2:=\min\{(1-N^2)\Pr, 2\Pr/L^2, 1 \}\quad\text{and}\quad
c_3:= \max \{(\Ra^2 \Pr)/(1-N^2), 2D^2,2 \}.
\]
Write
\[
y(t) =  |u(t)|^2 + |\o(t)|^2 + |\t(t)|^2\qquad \text{and}\qquad
\alpha(t) = \|u(t)\|^2 + \|\o(t)\|^2 +\|\t(t)\|^2
\]
Multiply \eqref{oszacowanie.energ.uomegateta} by $\exp (-c_3t)$ 
\[
\frac{d}{dt} \le y(t)e^{-c_3t} \pr  + c_2e^{-c_3t}\alpha(t) \leq 0
\]
and integrate from $0$ to some $s>0$ 
\[
y(s) + c_2 \int_0^s e^{c_3(s-t)}\alpha(t)\d t \leq e^{c_3s}y(0).
\]
Recall that 
\[
y(0) =(\leq) \left| u_0 \right |^2 +\left| \o_0 \right |^2 + \left| \t_0 \right |^2 .
\]
Fix $T>0$, the above inequality implies that
\begin{equation}\label{przestrzenie}
\begin{aligned}
u\in L^\infty (0,T;H_S)\cap L^2(0,T;V_S),\\
\o,\;\t\in L^\infty (0,T;H)\cap L^2(0,T;V)
\end{aligned}
\end{equation}
and 
\begin{equation}\label{oczacowania.norm}
\begin{aligned}
\|u\|^2_{L^\infty (0,T;H_S)}+\|\o\|^2_{L^\infty (0,T;H)}+\|\t\|^2_{L^\infty (0,T;H)}\leq e^{c_3T}\le|u_0|^2 +|\o_0|^2 + |\t_0|^2 \pr ,\\
\|u\|^2_{L^2(0,T;V_S)}+\|\o\|^2_{L^2(0,T;V)}+\|\t\|^2_{L^2(0,T;V)}\leq \frac{e^{c_3T}}{c_2}\le|u_0|^2 +|\o_0|^2 + |\t_0|^2 \pr.
\end{aligned}
\end{equation}

In order to prove the continuity of the functions 
\[
[0,T]\ni t \mapsto u(t) \in H_S\quad\text{and}\quad [0,T]\ni t \mapsto \o(t) \in H
\]
it is enough to show that $\dfrac{d u}{d t} \in L^2(0,T;V_S^\ast)$ and $\dfrac{d \o}{dt}\in L^2(0,T;V^\ast)$, see \cite[Thm 3, Chapter 5.9]{evans}.
We will show that $\dfrac{d u}{dt}\in L^2(0,T;V_S^\ast)$, the proof that $\dfrac{d \o}{dt}\in L^2(0,T;V)$ is similar.
From \eqref{row.u.2} we obtain
\[
\frac{1}{\Pr}\frac{d u}{dt}= -\frac{1}{\Pr} (u\cdot\nabla)u +\Delta u + 2 N^2\rot \o + \Ra \t e_2.
\]
Take $v\in L^2(0,T;V_S)$ then we have
\begin{equation}\label{oszacowanie.poch.u}
 \bigg._{V_S^\ast}\left\langle \frac{d u }{dt}, v\right \rangle _{V_S} = - b_S(u,u,v) - \Pr((u,v))  + 2\Pr N^2 (rot\o,v) + \Ra\Pr(\t e_2,v).
\end{equation}
We use the inequality \eqref{ladyzhenskaya} to show that
\[
|b_S(u,u,v)|= |b_S(u,v,u)| \leq \|u\|_{L^4}\|v\|\|u\|_{L^4}\leq k_4^2|u|\|u\|\|v\|.
\]
From \eqref{oszacowanie.poch.u} we get
\begin{multline*}
\int_0^T \bigg._{V_S^\ast}\left\langle \frac{du}{dt}, v \right\rangle_{V_S} \leq   k_4^2\|u\|_{L^\infty(0,T;H_S)}\|u\|_{L^2(0,T;V_S)}\|v\|_{L^2(0,T;V_S)} + \Pr \|u\|_{L^2(0,T;V_S)}\|v\|_{L^2(0,T;V_S)} \\
+2k_1\Pr N^2 \|\o\|_{L^2(0,T;V)}\|v\|_{L^2(0,T;V_S)}+ k_1\Ra\Pr \|\t\|_{L^2(0,T;V)}\|v\|_{L^2(0,T;V_S)},
\end{multline*}
where we used the Poincar\'{e} inequality in the last two terms.
Hence and by \eqref{przestrzenie}, we have 
\[
\left\|\dfrac{du}{dt}\right\|_{L^2(0,T;V_S^\ast)}\leq C,
\]
where $C$ depends only on $\Omega$, $u_0$, $\o_0$, $\t_0$ and $T>0$.

Recall that we have
\begin{equation}\label{row.pochodna.teta}
\frac{\partial \t }{\partial t}= -u\cdot\nabla\t +\Delta \t+ D\rot\o\cdot \nabla \t + D\frac{\partial \o}{\partial x_1} + u_2.
\end{equation}
We show that $\dfrac{d\t}{dt}\in L^2(0,T;D(A^{3/2}))$.
Every term from the right-hand side of \eqref{row.pochodna.teta} may be bounded in $L^2(0,T;V^\ast)$ as in \eqref{oszacowanie.poch.u} except from
\[
 D\rot\o\cdot \nabla \t.
\]
We write $X:=D(A^{3/2})$ for short.
Take $\f \in X$, integrate by parts and apply H\"{o}lder's inequality to~get
\begin{multline*}
_{X^\ast}\langle \rot \o \cdot \nabla \t, \f \rangle_X= \int_\Omega \le \rot \o\cdot \nabla \t \pr \f  = -  \int_\Omega \le \rot \o \cdot \nabla \f\pr \t \\
\leq |\t|\le \int _\Omega \left |\rot \o \cdot \nabla \f \right |^2\d x\pr ^\frac{1}{2} \leq |\t| \|\o\| \|\nabla\f\|_{L^\infty (\Omega)}.
\end{multline*}
In view of Theorem \ref{tw.ulamkowy.laplasjan}, $\nabla \f \in H^2(\Omega)^2$ so the Sobolev embedding yields 
\begin{equation}\label{nierownosc.X}
\|\nabla \f\|_{L^\infty (\Omega)} \leq \tilde{C}\|\nabla \f\|_{H^2(\Omega)} \leq \tilde{C}\|\f\|_{H^3(\Omega)} \leq C\|\f\|_X,
\end{equation}
where the last inequality follows from Theorem \ref{tw.ulamkowy.laplasjan}.
Summing up, we have
\[
_{X^\ast}\langle \rot \o\cdot \nabla \t, \f \rangle_X \leq C|\t|\|\o\|\|\f\|_X
\]
so 
\[
\|\rot\o\cdot\nabla \t\|_{X^\ast} \leq C|\t| \|\o\|.
\]
Therefore, we have 
\begin{multline*}
\|\rot \o \cdot \nabla\t\|_{L^2(0,T;X^\ast)} = \le \int_0^T \|\rot \o \cdot \nabla\t\|_{X^\ast}^2 \pr^\frac{1}{2} \leq C\le \int_0^T |\t|^2 \|\o\|^2 \pr ^\frac{1}{2} \\
\leq C \|\t\|_{L^\infty (0,T;H)}\|\o\|_{L^2(0,T;V)}.
\end{multline*}
The continuous embeddings $V^\ast\subset X^\ast$ yields the continuous embedding
\[
L^2(0,T;V^\ast) \subset L^2(0,T;X^\ast)
\]
and all the bounds on the components of \eqref{row.pochodna.teta}  made in the space $L^2(0,T;V^\ast)$ are valid in the space $L^2(0,T;X^\ast)$.
As a result, we get
\[
\left \|\frac{d \t}{d t} \right\|_{L^2(0,T;X^\ast)}\leq C,
\]
where $C$ depends only on $\Omega$, $T>0$ and initial conditions \eqref{war.pocz}.
\end{proof}
\section{Initial data in $H^1$}

\begin{df}
Let $T>0$,  $u_0\in V_S$, $\o_0\in V$ and $\t_0\in V$.
By a strong solution of problems \eqref{row.u.2}--\eqref{war.pocz} we mean a triple of functions $(u,\o,\t)$,
\begin{align*}
u&\in  L^2(0,T;D(A_S))\cap C([0,T],V_S) \cap W^{1,2}(0,T;H_S),\\
\o,\;\t&\in L^2(0,T;D(A))\cap C([0,T],V) \cap W^{1,2}(0,T;H),
\end{align*}
such that $u(0)=u_0$, $\o(0)=\o_0$, $\t(0)=\t_0$ and satisfying the following identities
\[
\frac{1}{\Pr}\le \frac{d}{dt}(u(t),\f) + b_S(u(t),u(t),\f)\pr + (-\Delta u(t), \f) = 2N^2 (\rot \o(t),\f) + \Ra (\t(t) e_2,\f)
\]
for every $\f\in H_S$,
\[
\frac{1}{\Pr}\le \frac{d}{dt}(\o(t),\psi) + b(u(t),\o(t),\psi)\pr + 4N^2 (\o(t),\psi) + \frac{1}{L^2}(-\Delta \o(t), \psi) = 2N^2 (\rot u(t),\psi)
\]
for every $\psi\in H$,
\[
\frac{d}{dt}(\t(t),\eta) + b(u(t),\t(t),\eta) + (-\Delta\t(t),\eta)  = -D (\rot \o(t) \cdot \nabla \t,\eta) + (u_2,\eta) + D \le\frac{\partial\o}{\partial x_1}(t),\eta\pr
\]
for every $\eta\in H$, in the sense of scalar distributions on $(0,T)$.
\end{df}
\begin{thm}\label{tw.istnienie.H1}
Let $u_0\in V_S$, $\o_0\in V$, $\t_0\in V$ and $T>0$.
There is a unique strong solution  $(u(t),\o(t),\t(t))$ of \eqref{row.u.2}--\eqref{war.pocz}.
Moreover, the strong solution depends continuously on the initial data, namely the following map is continuous
\[
V_S\times V\times V \ni (u_0,\o_0,\t_0) \mapsto (u,\o,\t)\in C([0,T],V_S\times V\times V).
\]
\end{thm}
\begin{proof}
The existence of solutions is based on the Galerkin-Faedo approximations as in the proof of Theorem \ref{tw.istnienie.L2}.
Since the assumptions of Theorem \ref{tw.istnienie.L2} are stronger than those of Theorem \ref{tw.istnienie.L2}, we may assume that $\o$, $\t\in L^\infty(0,T;H)\cap L^2(0,T;V)$.

First, we focus on $u$ and $\o$ equations. 
Multiply \eqref{row.u.2} by $A_Su$ and integrate over $\Omega$ 
\begin{equation}\label{row.entr.u}
\frac{1}{\Pr}\le \frac{\partial}{\partial t}u ,A_S u \pr   + \frac{1}{\Pr} b_S(u, u, A_S u) + | A_S u|^2 = 2N^2 \le \rot \o, A_Su \pr + \Ra \le\t e_2, A_S u \pr .
\end{equation} 
By \eqref{Agmon.S} and by H\"{o}lder's inequality with $(\infty,2,2)$ rates we have
\[
\Pr^{-1}|b_S(u,u,A_Su)|\leq \Pr^{-1}\|u\|_{L^\infty} \|u\||A_Su|\leq k_7\Pr^{-1}|u|^{1/2}\|u\||A_Su|^{3/2}\leq \frac{1}{6}|A_Su|^2 + c_1|u|^2\|u\|^4,
\]
where we used Young's inequality.
We estimate
\begin{gather*}
2N^2(\rot \o,A_Su) \leq 2N^2|\rot\o||A_Su|\leq \frac{1}{6}|A_S u|^2 +c_2 \|\o\|^2,\\
\Ra (\t e_2,A_S u) \leq \Ra|\t||A_Su|\leq \frac{1}{6}|A_Su|^2 + c_3|\t|^2.
\end{gather*}
From the above estimates and from \eqref{row.entr.u} we get
\begin{equation}\label{oszacowanie.entr.u}
\frac{d}{dt}\|u\|^2 + \Pr|A_Su|^2 \leq  2\Pr \le c_1|u|^2\|u\|^4 + c_2 \|\o\|^2 + c_3|\t|^2\pr.
\end{equation}

Multiply \eqref{row.omega.2} by $A\o$ and integrate over $\Omega$
\begin{equation}\label{row.entr.omega}
\frac{1}{\Pr}\le\frac{\partial}{\partial t}\o,A\o\pr + \frac{1}{\Pr}b(u,\o,A\o) + 4N^2(\o,A\o) + \frac{1}{L^2}|A\o|^2 = 2N^2 (\rot u, A\o).
\end{equation}
H\"{o}lder's inequality, \eqref{ladyzhenskaya}, \eqref{nier.grad} and Young's inequality yield
\[
|b(u,\o,A\o)| \leq \|u\|_{L^4}\|\nabla \o\|_{L^4} |A\o| \leq k_4k_5 |u|^{1/2}\|u\|^{1/2}\|\o\|^{1/2}|A\o|^{3/2}\leq \frac{1}{4L^2}|A\o|^2 +c_4|u|^2\|u\|^2\|\o\|^2.
\]
We have
\[
2N^2 (\rot u, A\o)\leq 2N^2|\rot u||A\o|\leq \frac{1}{4L^2}|A\o|^2 + c_5 \|u\|^2.
\]
Hence, we may transform \eqref{row.entr.omega} into
\begin{equation}\label{oszacowanie.entr.omega}
\frac{d}{dt}\|\o\|^2+8N^2 \|\o\|^2 +\frac{\Pr}{L^2}|A\o|^2 \leq 2\Pr \le c_4 |u|^2\|u\|^2\|\o\|^2 + c_5\|u\|^2\pr.
\end{equation}
We add \eqref{oszacowanie.entr.u} and \eqref{oszacowanie.entr.omega} to get
\begin{multline}\label{oszacowanie.entr.u.omega}
\frac{d}{dt} \le \|u\|^2 + \|\o\|^2 \pr + \Pr |A_Su|^2 + \frac{\Pr}{L^2}|A\o|^2 \\
\leq C_1 \le |u|^2\|u\|^4 + \|\o\|^2 +|\t|^2 +  |u|^2\|u\|^2\|\o\|^2  + \|u\|^2\pr,
\end{multline}
where $C_1:=2\Pr\cdot\max_{i=1\ldots 5}\{c_i\}$.
We drop out the terms $|A_Su|^2$ and $|A\o|^2$ in \eqref{oszacowanie.entr.u.omega}, and obtain
\begin{equation}\label{oszacowanie.nier.Gron}
\frac{d}{dt} \le \|u\|^2 + \|\o\|^2 \pr \leq C_1 \le \|u\|^2 + \|\o\|^2\pr \le |u|^2 \|u\|^2 + |u|^2\|\o\|^2 +1\pr + C_1 |\t|^2.
\end{equation}
Let us denote
\begin{gather*}
y(t):=\|u(t)\|^2 +\|\o(t)\|^2,\qquad\alpha(t):=C_1 |\t(t)|^2,\\
\beta(t):=C_1 \le |u(t)|^2 \|u(t)\|^2 + |u(t)|^2\|\o(t)\|^2 +1\pr,
\end{gather*}
then \eqref{oszacowanie.nier.Gron} turns into 
\begin{equation}\label{nier.do.Gronwalla}
y^\prime(s)\leq \alpha(s) + \beta(s)y(s).
\end{equation}
Since $\alpha$, $\beta\in L^1(0,T)$, we multiply \eqref{nier.do.Gronwalla} by $\exp\le - \int_0^s\beta(\tau)\d \tau \pr$ and integrate from $0$ to $t$
\begin{equation}\label{nier.po.Gronwalu.2}
y(t)\leq y(0)\exp \le \int_0^t\beta(\tau) \d \tau \pr + \int_0^t\alpha(s)\exp\le\int_s^t \beta(\tau)\d\tau\pr\d s \qquad t\in [0,T].
\end{equation}
The right hand-side of \eqref{nier.po.Gronwalu.2} is bounded, for every $t\in [0,T]$, so
\begin{equation}
u\in L^\infty(0,T;V_S)\qquad\text{and}\qquad\o\in L^\infty(0,T;V).
\end{equation}
Now, we integrate \eqref{oszacowanie.entr.u.omega} from $0$ to $T$ and get
\begin{equation}\label{ogr.Aomega.Au}
\int_0^T \Pr|A_Su|^2 + \frac{\Pr}{L^2}|A\o|^2 \leq C,
\end{equation}
for some constant $C>0$ depending only on $\Omega$, $T>0$ and intial data.
Thus 
\[
u\in L^2(0,T;D(A_S))\qquad\text{and}\qquad\o\in L^2(0,T;D(A)).
\]

Let us consider $\t$ equation.
Multiply \eqref{row.teta} by $A\t$ and integrate over $\Omega$
\begin{equation}\label{row.entr.teta}
\le \frac{\partial}{\partial t} \t,A\t\pr + b(u,\t,A\t) + |A\t|^2 = D\int_\Omega(\rot\o\cdot\nabla\t)A\t + D\le \frac{\partial \o}{\partial x_1},A\t\pr + (u_2,A\t).
\end{equation}
Proceeding as above, we have
\begin{gather*}
|b(u,\t,A\t)|\leq \|u\|_{L^4}\|\nabla \t\|_{L^4}|A\t|^2 \leq k_4k_5|u|^{1/2}\|u\|^{1/2}\|\t\|^{1/2}|A\t|^{3/2}\leq\frac{1}{8}|A\t|^2 + c_7|u|^2\|u\|^2\|\t\|^2,\\
D\le \frac{\partial \o}{\partial x_1},A\t\pr \leq D\|\o\||A\t|\leq \frac{1}{8}|A\t|^2 + c_8\|\o\|^2,\\
(u_2,A\t)\leq |u||A\t|\leq \frac{1}{8}|A\t|^2 + c_9|u|^2.
\end{gather*}
We use H\"{o}lder's inequality with $(4,4,2)$ rates, \eqref{nier.grad} and Young's inequality
\begin{multline}\label{oszac.rot.grad}
D\int_\Omega (\rot \o\cdot\nabla \t)A\t \leq D \|\nabla \o\|_{L^4}\|\nabla\t\|_{L^4}|A\t|\\
\leq Dk_5^2\|\o\|^{1/2}|A\o|^{1/2}\|\t\|^{1/2}|A\t|^{3/2}\leq \frac{1}{8}|A\t|^2 +c_{10}\|\o\|^2|A\o|^2\|\t\|^2.
\end{multline}
The equation \eqref{row.entr.teta} is transformed into
\begin{equation}\label{oszacowanie.entr.teta}
\frac{d}{dt}\|\t\|^2 + |A\t|^2 \leq C_2\le |u|^2\|u\|^2\|\t\|^2 + \|\o\|^2|A\o|^2\|\t\|^2 + \|\o\|^2 +|u|^2\pr,
\end{equation}
where $C_2:=2\max_{i=7,8,9,10}\{c_i\}$.
We argue as in \eqref{nier.po.Gronwalu.2} and \eqref{ogr.Aomega.Au} to obtain
\[
\t\in L^\infty (0,T;V)\cap L^2(0,T;D(A)).
\]

We will show that $\dfrac{d\t}{dt}\in L^2(0,T;H)$.
By \cite[Thm 3, Chapter 5.9]{evans}, since $\t\in L^2(0,T;D(A))$, this yields $\t\in C([0,T],V)$ (the same reasoning for $u$ and $\t$).
We only show $\dfrac{d\t}{dt}\in L^2(0,T;H)$ because $u$ and $\o$ cases are similar.
From \eqref{row.teta} we have
\begin{equation}\label{row.poch.teta.2}
\frac{d\t}{dt} = -u\cdot\nabla \t  -A\t +D\rot\o\cdot \nabla \t +D\frac{\partial \o}{\partial x_1} + u_2.
\end{equation} 
We use \eqref{ladyzhenskaya} and \eqref{nier.grad} to get
\[
|u\cdot\nabla\t| \leq \|u\|^2_{L^4}\|\nabla \t\|^2_{L^4} \leq k_4k_5 \|u\|^2\|\t\||A\t|
\]
and, similarly,
\[
|\rot\o\cdot\nabla\t|\leq \|\nabla \o\|^2_{L^4}\|\nabla \t\|^2_{L^4}\leq k_5^2\|\o\||A\o|\|\t\||A\t|.
\]
From \eqref{row.poch.teta.2} we obtain
\[
\begin{split}
\left\| \frac{d\t}{dt}\right\|_{L^2(0,T;H)}&\leq \le \int_0^T |u\cdot\nabla\t|^2 \pr^{1/2} + \le \int_0^T |\rot\o\cdot\nabla\t|^2 \pr^{1/2} + \|\t\|_{L^2(0,T;D(A))} \\
&+ D\|\o\|_{L^2(0,T;V)} + \|u\|_{L^2(0,T;H_S)}\\
&\leq C\big (\|u\|^2_{L^\infty(0,T;V_S)}\|\t\|_{L^\infty(0,T;V)}\|\t\|_{L^2(0,T;D(A))}\\
&+\|\o\|_{L^\infty (0,T;V)}\|\o\|_{L^2(0,T;D(A))}\|\t\|_{L^\infty (0,T;V)}\|\t\|_{L^2(0,T;D(A))} \\
&+ \|\t\|_{L^2(0,T;D(A))} + \|\o\|_{L^2(0,T;V)} + \|u\|_{L^2(0,T;H_S)}\big ),
\end{split}
\]
for a suitable $C>0$.
Thus $\dfrac{d\t}{dt}$ is bounded in ${L^2(0,T;H)}$.

Now, we show that the solution is unique and depends continuously on initial data.
Let $(\ov u_0,\ov \o_0, \ov \t_0)$ and $(\wh u_0, \wh\o_0,\wh\t_0)$ be two inital states, and $(\ov u,\ov\o,\ov\t)$, $(\wh u,\wh\o,\wh\t)$ be two  corresponding strong solutions.
If we set $(u,\o,\t)=(\ov u - \wh u,\ov\o-\wh\o,\ov \t - \wh\t)$, then $(u,\o,\t)$ satisfies
\begin{gather}
\frac{1}{\Pr}\frac{du}{dt} -\Delta u + \nabla p = 2N^2\rot\o +\Ra \t e_2 -\frac{1}{\Pr}(\ov u\cdot \nabla)\ov u + \frac{1}{\Pr}(\wh u \cdot \nabla)\wh u,\label{row.jedn.u}\\
\div u = 0,\notag\\
\frac{1}{\Pr}\frac{d\o}{dt} +4N^2\o - \frac{1}{L^2}\Delta \o = 2N^2\rot u -\frac{1}{\Pr}\ov u\cdot \nabla\ov \o + \frac{1}{\Pr}\wh u \cdot \nabla\wh \o,\label{row.jedn.omega}\\
\frac{d\t}{dt} -\Delta \t = D\rot \ov\o \cdot \nabla\ov\t - D\rot\wh\o\cdot\nabla\wh\t + D\frac{\partial \o}{\partial x_1} +u_2 - \ov u\cdot \nabla\ov \t + \wh u \cdot \nabla\wh \t.\label{row.jedn.teta}
\end{gather}
We have
\[
(\ov u\cdot \nabla)\ov u - (\wh u \cdot \nabla)\wh u = (\wh u\cdot\nabla )u + (u\cdot\nabla)\ov u,
\]
so by multiplying the equation \eqref{row.jedn.u} by $A_Su$ and integrating over $\Omega$ we get
\[
\frac{1}{2\Pr}\frac{d}{dt}\|u\|^2 +|A_Su|^2= 2N^2(\rot\o,A_Su) +\Ra(\t e_2,A_Su) - \frac{1}{\Pr}b_S(\wh u, u ,A_Su) - \frac{1}{\Pr}b_S(u,\ov u,A_Su).
\]
We estimate the right hand-side (term by term) like before
\begin{multline}\label{jedn.u}
\frac{d}{dt}\|u\|^2 +2\Pr|A_Su|^2\\
\leq \frac{\Pr}{4}|A_Su|^2 + d_1\|\o\|^2 + \frac{\Pr}{4}|A_Su|^2+\wh d_2 |\t|^2 + C|A_S\wh u|\|u\||A_Su| + k_7|u|^{1/2}\|\ov u\||A_Su|^{3/2}\\
\leq \frac{\Pr}{2}|A_Su|^2 +d_1\|\o\|^2 + d_2\|\t\|^2 +  \frac{\Pr}{4}|A_Su|^2 + d_3|A_S\wh u|^2\|u\|^2 + \frac{\Pr}{4}|A_Su|^2 + \wh d_4 |u|^2 \|\ov u\|^4 \\
\leq \Pr|A_Su|^2 + d_1\|\o\|^2 + d_2\|\t\|^2 + d_3 |A_S\wh u|^2\|u\|^2 + d_4\|u\|^2 \|\ov u\|^4 ,
\end{multline}
where we used the embedding $D(A_S)\subset L^\infty$ and switching from $\wh d_i$ to $d_i$ means that the Poincar\'e inequality was used.

We multiply \eqref{row.jedn.omega} by $A\o$ and integrate over $\Omega$ 
\[
\frac{d}{dt}\|\o\|^2 +\frac{2\Pr}{L^2}|A\o|^2 + 8N^2\Pr\|\o\|^2= 4N^2\Pr(\rot u,A\o) - 2b(\wh u, \o,A\o) - 2b(u,\ov \o,A\o).
\]
We estimate as before
\begin{multline}\label{jedn.omega}
\frac{d}{dt}\|\o\|^2 +\frac{2\Pr}{L^2}|A\o|^2 + 8N^2\Pr\|\o\|^2\\
\leq \frac{\Pr}{4L^2}|A\o|^2 + d_5\|u\|^2 + C|A_S\wh u|\|\o\||A\o| + C\|u\|_{L^4}\|\nabla\ov \o\|_{L^4}|A\o|\\
\leq \frac{\Pr}{4L^2}|A\o|^2 + d_5\|u\|^2 +\frac{\Pr}{4L^2}|A\o|^2 + d_6|A_S\wh u|^2\|\o\|^2 +\frac{\Pr}{4L^2}|A\o|^2 + d_7\|u\|^2|A\ov\o|^2\\
=\frac{3\Pr}{4L^2}|A\o|^2 + d_5\|u\|^2+ d_6|A_S\wh u|^2\|\o\|^2+ d_7\|u\|^2|A\ov\o|^2,
\end{multline}
where we used the embeddings $D(A_S)\subset L^\infty$, $H^1\subset L^4$ and regularity theorems.

Similarly, we multiply \eqref{row.jedn.teta} by $A\t$ and integrate over $\Omega$ 
\begin{multline*}
\frac{1}{2}\frac{d}{dt}\|\t\|^2 + |A\t|^2 = D \le \frac{\partial \o}{\partial x_1},A\t\pr +2(u_2,A\t) \\
 +D\int_\Omega(\rot \o\cdot\nabla\ov\t)A\t + D\int_\Omega(\rot\wh\o\cdot\nabla\t)A\t - b(\wh u, \t,A\t) - b(u,\ov\t,A\t) .
\end{multline*}
We estimate as above
\begin{multline}\label{jedn.teta}
\frac{d}{dt}\|\t\|^2 + 2|A\t|^2 \leq \frac{1}{6}|A\t|^2 +  d_8\|\o\|^2 + \frac{1}{6}|A\t|^2 + \wh d_9|u|^2 \\
+  C\|\o\|^{1/2}|A\o|^{1/2}\|\ov\t\|^{1/2}|A\ov\t|^{1/2}|A\t| + C\|\wh\o\|^{1/2}|A\wh\o|^{1/2}\|\t\|^{1/2}|A\t|^{3/2} \\
+ 2\|u\|_{L^4}\|\nabla\ov\t\|_{L^4}|A\t| + C|A_S\wh u|\|\t\||A\t|\\
\leq \frac{1}{3}|A\t|^2 + d_8\|\o\|^2 + d_9\|u\|^2 + \frac{1}{6}|A\t|^2 + C\|\o\||A\o|\|\ov\t\||A\ov\t| \\
+\frac{1}{6}|A\t|^2+d_{11}\|\wh\o\|^2|A\wh\o|^2\|\t\|^2 + C\|u\||A\ov\t||A\t| + \frac{1}{6}|A\t|^2 + d_{13}|A_S \wh u|^2\|\t\|^2\\
\leq |A\t|^2 +\frac{\Pr}{4L^2}|A\o|^2 + d_8\|\o\|^2 +d_9\|u\|^2 +  d_{10}\|\o\|^2\|\ov\t\|^2|A\ov\t|^2 + d_{11}\|\wh\o\|^2|A\wh\o|^2\|\t\|^2 \\
+d_{12}\|u\|^2|A\ov\t|^2 + d_{13}|A_S \wh u|^2\|\t\|^2,
\end{multline}
where we used the embeddings $H^1\subset L^4$, $D(A)\subset L^\infty$ and regularity results.
We add \eqref{jedn.u}, \eqref{jedn.omega} and \eqref{jedn.teta} and make some simple calculations
\begin{multline}\label{jedn.oszac.wsp}
\frac{d}{dt}\le \|u\|^2 + \|\o\|^2 + \|\t\|^2\pr + \Pr |A_S u|^2 + \frac{\Pr}{L^2}|A\o|^2 + 8N^2\Pr\|\o\|^2+ |A\t|^2 \\
\leq D \le \|u\|^2 + \|\o\|^2+\|\t\|^2\pr 
\le |A_S \wh u|^2 + \|\ov u\|^4 +|A\ov\o|^2+ |A\ov\t|^2 + \|\ov\t\|^2|A\ov\t|^2+ \|\wh\o\|^2|A\wh\o|^2+ 2 \pr ,
\end{multline}
where $D$ is the maximum of $d_1,\ldots, d_{13}$.

Let us denote 
\[
y:= \|u\|^2+\|\o\|^2+\|\t\|^2
\]
and 
\[
\alpha:= D\le |A_S \wh u|^2 + \|\ov u\|^4 +|A\ov\o|^2+|A\ov\t|^2+ \|\ov\t\|^2|A\ov\t|^2+ \|\wh\o\|^2|A\wh\o|^2+ 2\pr.
\]
Omitting some terms on the left hand-side of \eqref{jedn.oszac.wsp} yields
\[
y^\prime(t)\leq \alpha(t)y(t)
\]
and so
\begin{equation}\label{nier.jedn.ciaglosc}
y(t)\leq y(0)\exp{\int_0^T\alpha(s)\; \d s},\qquad t\in [0,T].
\end{equation}
Using \eqref{oczacowania.norm}, \eqref{nier.po.Gronwalu.2}, \eqref{ogr.Aomega.Au} and  \eqref{oszacowanie.entr.teta} we may find functions $g_1$, $g_2:\R^3 \to [0,\infty)$ which map bounded subsets of $\R^3$ onto bounded subsets of $[0,\infty)$, $g_1$, $g_2$ are increasing with respect to each variable and satisfy
\begin{equation}\label{oszacowanie.norm.3}
\begin{aligned}
\|u\|^2_{L^\infty(0,T;V_S)} + \|\o\|^2_{L^\infty(0,T;V)}+ \|\t\|^2_{L^\infty(0,T;V)}&\leq g_1(\|u_0\|,\|\o_0\|,\|\t_0\|),\\
\|u\|^2_{L^2(0,T;D(A_S))} + \|\o\|^2_{L^2(0,T;D(A))} + \|\t\|^2_{L^2(0,T;V)} &\leq g_2(\|u_0\|,\|\o_0\|,\|\t_0\|),
\end{aligned}
\end{equation}
where $(u,\o,\t)$ is a strong solution starting at $(u_0,\o_0,\t_0)$.
We use the regularity theorem for Stokes operator
\[
\|\ov u\|^2 \leq \|\ov u\|^2_{H^2} \leq C|A_S u|^2
\]
to get
\[
\alpha \leq C \le\|\ov u\|^2 + \|\ov\t\|^2 + \|\wh \o\|^2 +1\pr \le |A_S \wh u|^2 + |A\ov \o|^2 +|A\wh\o|^2+|A\ov\t|^2+1 \pr .
\]
We write $\wh g_1$ instead of $g_1(\|\wh u_0\|,\|\wh\o_0\|,\|\wh\t_0\|)$ for short and symbols $\ov g_1$, $\wh g_2$, $\ov g_2$ have similar meanings.
By \eqref{nier.jedn.ciaglosc} and \eqref{oszacowanie.norm.3}, we have
\begin{equation}\label{oszacowanie.koncowe}
\|u(t)\|^2+\|\o(t)\|^2+\|\t(t)\|^2 \leq \le\|u_0\|^2+\|\o_0\|^2+\|\t_0\|^2\pr\exp \le C\le T + \ov g_1+ \wh g_1\pr\le \ov g_2 + \wh g_2 +T\pr \pr.
\end{equation}
When $\ov u_0 = \wh u_0$, $\ov \o_0 = \wh\o_0$ and $\ov\t_0 = \wh\t_0$, \eqref{oszacowanie.koncowe} implies that the strong solution is unique.

To prove continuous dependence on initial data take $\e>0$ and $(\ov u_0,\ov\o_0,\ov\t_0)\in V_S\times V\times V$.
Let us denote
\[
M:=\max\left\{ g_1\le\|\ov u_0\| +1,\|\ov\o_0\| +1,\|\ov\t_0\| +1\pr, g_2\le\|\ov u_0\| +1,\|\ov\o_0\| +1,\|\ov\t_0\| +1\pr\right\} +T
\] and put
\[
\delta:= \min\left \{1, \e \exp\le -\frac{1}{2}C\le 4M^2\pr\pr\right\}.
\]
If
\[
\| (\ov u_0, \ov\o_0,\ov\t_0) - (\wh u_0,\wh\o_0,\wh\t_0)\|_{V_S\times V\times V}<\delta
\]
then
\[
\sup_{t\in [0,T]} \|u(t)\|^2 +\|\o(t)\|^2+\|\t(t)\|^2 \leq \e^2.\qedhere
\]
\end{proof}

\section{Mixed initial conditions}
\begin{thm}\label{tw.mieszane.war}
Let $T>0$,  $u_0\in H_S$, $\t_0\in H$ and $\o_0\in V$.
There is a solution  $(u,\o,\t)$ to \eqref{row.u.2}--\eqref{war.pocz} in the sense that
\begin{align*}
u&\in L^2(0,T;V_S)\cap C([0,T],H_S) \cap W^{1,2}(0,T;V_S^\ast),\\
\o&\in L^2(0,T;D(A))\cap C([0,T],V)\cap W^{1,2}(0,T;H),\\
\t&\in L^2(0,T;V)\cap C([0,T];H)\cap W^{1,2}(0,T;V^\ast)
\end{align*}
such that $u(0)=u_0$, $\o(0)=\o_0$, $\t(0)=\t_0$ and satisfying the following identities
\begin{equation}\label{row.u.mieszane}
\frac{1}{\Pr}\le \frac{d}{dt}(u(t),\f) + b_S(u(t),u(t),\f)\pr + (\nabla u(t),\nabla \f) = 2N^2 (\rot \o(t),\f) + \Ra (\t(t) e_2,\f)
\end{equation}
for every $\f\in V_S$,
\begin{equation}\label{row.omega.mieszane}
\frac{1}{\Pr}\le \frac{d}{dt}(\o(t),\psi) + b(u(t),\o(t),\psi)\pr + 4N^2 (\o(t),\psi) + \frac{1}{L^2}(-\Delta\o(t), \psi) = 2N^2 (\rot u(t),\psi)
\end{equation}
for every $\psi\in H$,
\begin{equation}\label{row.teta.mieszane}
\frac{d}{dt}(\t(t),\eta) + b(u(t),\t(t),\eta) + (\nabla\t(t),\nabla\eta)  = D (\t(t),\rot \o(t) \cdot \nabla \eta) +D \le\frac{\partial\o}{\partial x_1}(t),\eta\pr + (u_2,\eta)
\end{equation}
for every $\eta\in V$, in the sense of scalar distributions on $(0,T)$.
Moreover, if $N^2L^2<1$ then such solution is unique and depends continuously on the initial conditions, i.e. the following map is continuous
\[
H_S\times V\times H \ni (u_0,\o_0,\t_0) \mapsto (u,\o,\t)\in C([0,T],H_S\times V\times H).
\]
\end{thm}
\begin{proof}
The proof is similar in spirit to that of  Theorem \ref{tw.istnienie.L2} and Theorem \ref{tw.istnienie.H1}.
We restrict ourselves to present the sketch of the proof.

In view of Theorem \ref{tw.istnienie.L2}, we know that \eqref{warunk.ist.L2} holds.
First, we multiply \eqref{row.omega.2} by $A\o$ and integrate over $\Omega$ to get \eqref{oszacowanie.entr.omega}.
We use the kind of Gronwall's type inequality like in \eqref{nier.po.Gronwalu.2} and obtain 
\[
\o\in L^2(0,T;D(A))\cap L^\infty (0,T;V).
\]
Then we can show that $\dfrac{d\o}{dt}\in L^2(0,T;H)$ by showing that each term of \eqref{row.omega.2} belongs to $L^2(0,T;H)$.
Indeed, let us just mark that 
\[
A\o\in L^2(0,T;H)
\]
and 
\[
u\cdot \nabla \o \in L^2(0,T;H),
\]
because
\[
\|u\cdot\nabla \o\|_{L^2(0,T;H)}\leq C \|u\|_{L^\infty(0,T;H_S)}\|\o\|_{L^\infty(0,T;V)}.
\]

Now, we turn to $\t$ equation.
All we have to do is to show that $\dfrac{d \t}{dt}\in L^2(0,T;V^\ast)$.
We focus our attention on proving that $\rot\o\cdot\nabla\t \in L^2(0,T;V^\ast)$, since all the other terms from the right hand-side of \eqref{row.pochodna.teta} belong to $L^2(0,T;V^\ast)$ in a standard way.
Take $\f\in V$, we inegrate by parts, use H\"{o}lder's inequality \eqref{ladyzhenskaya} and \eqref{nier.grad}
\[
_V\langle \rot \o\cdot \nabla \t,\f\rangle_V = -\int_\Omega (\rot \o\cdot\nabla\f) \t \leq \|\rot \o \|_{L^4}\|\f\|\|\t\|_{L^4} \leq k_4k_5\|\o\|^{1/2}|A\o|^{1/2}|\t|^{1/2}\|\t\|^{1/2}\|\f\|.
\]
Thus, we end with
\[
\begin{split}
\|\rot\o\cdot\nabla\t\|^2_{L^2(0,T;V^\ast)}&\leq C\int_\Omega \|\o\||A\o||\t|\|\t\|\\
&\leq C\|\o\|_{L^\infty(0,T;V)}\|\o\|_{L^2(0,T;D(A))}\|\t\|_{L^\infty(0,T;H)}\|\t\|_{L^2(0,T;V)}
\end{split}
\]
what proves the regularity of solutions as claimed.

Now, let us assume that $N^2L^2<1$ and we have two solutions $(\ov u, \ov \o, \ov \t)$ and $(\ov u, \ov \o, \ov \t)$ (in the sense of Theorem \ref{tw.mieszane.war}) starting at $(\ov u_0, \ov \o_0,\ov \t_0)$ and $(\wh u_0,\wh \o_0, \wh \t_0)\in H_S\times V \times H$, respectively.
We set $(u,\o,\t) = (\ov u - \wh u , \ov \o - \wh \o, \ov \t - \wh \t)$ and from \eqref{row.u.mieszane}, \eqref{row.omega.mieszane} and \eqref{row.teta.mieszane} we obtain
\begin{multline}\label{nier.mieszane}
\frac{1}{2}\frac{d}{dt}\le \frac{1}{\Pr}\le |u|^2 + \|\o\|^2\pr + |\t|^2 \pr + \|u\|^2 + 4N^2\|\o\|^2 + \frac{1}{L^2}|A\o|^2+ \|\t\|^2 \\
= 2N^2 (\rot \o,u) + \Ra (\t e_2, u) - \frac{1}{\Pr} b_S(u, \ov u, u) \\
+2N^2(\rot u , A\o) - \frac{1}{\Pr}\le b(u,\ov \o,A\o) + b(\wh u, \o,A\o) \pr \\
+ D\le \frac{\partial \o}{\partial x_1} ,\t\pr + (u_2,\t) +  D\int_\Omega(\rot \o\cdot\nabla\ov\t)\t  - b(u,\ov\t,\t).
\end{multline}
We estimate every term of the rhs of \eqref{nier.mieszane}
\begin{gather*}
2N^2(rot\o,u)\leq 4N^2\|\o\|^2 +\frac{N^2}{4}|u|^2,\quad \Ra(\t e_2,u) \leq k_1\le |\t|^2 + |u|^2\pr,\\
-\frac{1}{\Pr}b_S(u,\ov u, u ) \leq \frac{1}{\Pr}\|u\|_{L^4}\|\ov u\| \|u\|_{L^4}\leq C|u|\|u\|\|\ov u\|\leq \frac{1-L^2N^2}{4}\|u\|^2 + k_2|u|^2\|\ov u\|^2,\\
2N^2 (\rot u, A\o)\leq L^2N^2\|u\|^2 + \frac{N^2}{L^2}|A\o|^2,\\
-\frac{1}{\Pr}b(u,\ov\o,A\o)\leq \frac{1}{\Pr}\|u\|_{L^4}\|\nabla\ov\o\|_{L^4}|A\o|\leq C |u|^{1/2}\|u\|^{1/2}\|\ov\o\|^{1/2}|A\ov\o|^{1/2}|A\o| \\
\leq \frac{1-N^2}{4L^2}|A\o|^2 + C^\prime |u|\|u\|\|\ov\o\||A\ov\o|\leq \frac{1-N^2}{4L^2}|A\o|^2 + \frac{1-L^2N^2}{4}\|u\|^2 + k_3|u|^2\|\ov\o\|^2|A\ov\o|^2,\\
-\frac{1}{\Pr}b(\wh u,\o,A\o) \leq C|\wh u|^{1/2}\|\wh u\|^{1/2}\|\o\|^{1/2}|A\o|^{3/2}\leq\frac{1-N^2}{4L^2}|A\o|^2 + k_4|\wh u|^2\|\wh u\|^2\|\o\|^2,\\
D\le \frac{\partial \o}{\partial x_1},\t \pr \leq k_5\le \|\o\|^2+ |\t|^2 \pr,\quad (u_2,\t)\leq k_6\le |u|^2 +|\t|^2\pr,\\
D\int_\Omega (\rot\o \cdot \nabla \ov\t)\t\leq C\|\nabla\o\|_{L^4}\|\ov\t\|\|\t\|_{L^4}\leq 2\min\left\{\frac{1-N^2}{4L^2},\frac{1}{4}\right\}|A\o|\|\t\| + C \|\o\||\t|\|\ov\t\|^2\\
\leq \frac{1-N^2}{4L^2}|A\o|^2+\frac{1}{4}\|\t\|^2 + k_7\le \|\o\|^2 + |\t|^2 \pr \|\ov\t\|^2,\\
-b(u,\ov\t,\t)\leq 2 \frac{1-L^2N^2}{4}\|u\|\|\t\|+ C|u||\t|\|\ov\t\|^2\leq \frac{1-L^2N^2}{4}\|u\|^2 +\frac{1}{4}\|\t\|^2 + k_8\le |u|^2 +|\t|^2 \pr \|\ov\t\|^2.
\end{gather*}
Hence and from \eqref{nier.mieszane} we get
\[
\frac{d}{dt}\le  |u|^2 + \|\o\|^2 + |\t|^2 \pr \leq C\le |u|^2 + \|\o\|^2 + |\t|^2\pr\le 1+ \|\ov u\|^2+ \|\ov\o\|^2|A\ov\o|^2+\|\ov\t\|^2+|\wh u |^2\|\wh u\|^2\pr, 
\]
for a suitable constant $C$, what proves the assertion.
\end{proof}

\begin{thm}\label{tw.mieszane.war2}
Let $T>0$,  $u_0\in V_S$, $\o_0\in V$ and $\t_0\in H$.
There is a unique solution  $(u,\o,\t)$ to \eqref{row.u.2}--\eqref{war.pocz} in the sense that
\begin{align*}
u&\in L^2(0,T;D(A_S))\cap C([0,T],V_S) \cap W^{1,2}(0,T;H_S),\\
\o&\in L^2(0,T;D(A))\cap C([0,T],V)\cap W^{1,2}(0,T;H),\\
\t&\in L^2(0,T;V)\cap C([0,T];H)\cap W^{1,2}(0,T;V^\ast)
\end{align*}
such that $u(0)=u_0$, $\o(0)=\o_0$, $\t(0)=\t_0$ and satisfying the following identities
\[
\frac{1}{\Pr}\le \frac{d}{dt}(u(t),\f) + b_S(u(t),u(t),\f)\pr + (-\Delta u(t), \f) = 2N^2 (\rot \o(t),\f) + \Ra (\t(t) e_2,\f)
\]
for every $\f\in H_S$,
\[
\frac{1}{\Pr}\le \frac{d}{dt}(\o(t),\psi) + b(u(t),\o(t),\psi)\pr + 4N^2 (\o(t),\psi) + \frac{1}{L^2}(-\Delta\o(t), \psi) = 2N^2 (\rot u(t),\psi)
\]
for every $\psi\in H$,
\[
\frac{d}{dt}(\t(t),\eta) + b(u(t),\t(t),\eta) + (\nabla\t(t),\nabla\eta)  = D (\t(t),\rot \o(t) \cdot \nabla \eta) +D \le\frac{\partial\o}{\partial x_1}(t),\eta\pr + (u_2,\eta)
\]
for every $\eta\in V$, in the sense of scalar distributions on $(0,T)$.
Moreover, the following map is continuous
\[
V_S\times V\times H \ni (u_0,\o_0,\t_0) \mapsto (u,\o,\t)\in C([0,T],V_S\times V\times H).
\]
\end{thm}
\begin{proof}
The proof is similar to the proofs of the previous theorems and so it is left to the reader.
\end{proof}

\end{document}